\documentclass[reqno]{amsart}
\usepackage{amsmath}

\usepackage{amsfonts}
\usepackage{amssymb}
\usepackage{hyperref}
\usepackage{xcolor}
\usepackage[a4paper,top=3.5cm,bottom=3cm,left=2.5cm,right=2.5cm]{geometry}

\begin{document}
\title[Existence results on $b$-metric space]{Existence results for a generalized fractional boundary value problem in $b$-metric space}

\author[F. Haddouchi]{Faouzi Haddouchi}

\address{
Department of Physics, University of Sciences and Technology of
Oran-MB, Oran, Algeria
\newline
And
\newline
Laboratory of Fundamental and Applied Mathematics of Oran (LMFAO), University of Oran 1, Oran, Algeria}
\email{fhaddouchi@gmail.com}

\begin{abstract}
This paper is concerned with a class of nonlinear boundary value problem involving fractional derivative in the $\varphi$-Riemann-Liouville sense. Some Properties of the Green's function for this problem are mentioned. By means of the Banach contraction principle in $b$-metric space and the technique of the $\gamma$–$\psi$ Geraghty contractive maps, existence and uniqueness results are obtained. Two examples are given to support the theoretical results.
\end{abstract}

\maketitle \numberwithin{equation}{section}
\newtheorem{theorem}{Theorem}[section]
\newtheorem{lemma}[theorem]{Lemma}
\newtheorem{definition}[theorem]{Definition}
\newtheorem{proposition}[theorem]{Proposition}
\newtheorem{corollary}[theorem]{Corollary}
\newtheorem{remark}[theorem]{Remark}
\newtheorem{exmp}{Example}[section]

\section{Introduction\label{sec:1}}

In recent years, the theory of fractional calculus has attracted considerable interest in mathematics and also in many applications, such as physics, mechanics, chemistry, engineering, etc. We refer to \cite{Kilbas,Debnath,Gambo,Jarad2,Mainardi,Mall,Qu,Rostami,Zhou,Podlubny} and references therein.

Many works have studied fractional differential equations using several definitions of a
fractional derivative, we cite here the works \cite{Seemab,Vivek,Had,Almeida1,Almeida2,Hilfer, Afshari2,Samko,
Jarad,Katugampola,Sousa1,Sousa2}.

In \cite{Karapinar}, the notion of generalized $\alpha-\psi$-Geraghty contractive
($\alpha-\psi$-GC) type mappings in complete $b$-metric spaces was introduced. Other recent works on this notion can be founded in \cite{Afshari10,Afshari3,Afshari4,Afshari5,Afshari1,Afshari11,Afshari7,Samet,Aydi}.\\

Afshari et.al., in \cite{Afshari8} studied the existence of positive solutions for
the generalized fractional BVPs of the form

\begin{equation}\label{eq01}
       \begin{cases}^{C}D^{i,\kappa}_{0+}u(t)+f(t,u(t))=0,\  t \in (0,1),\\
    u(0) = u(1) =0,
       \end{cases}
       \end{equation}
and
\begin{equation}\label{eq02}
       \begin{cases}^{C}D^{i,\kappa}_{0+}u(t)+f(t,u(t))=0,\  t \in (0,1),\\
    u(0)+u^{\prime}(0) =0,\ u(1)+u^{\prime}(1) =0,
       \end{cases}
       \end{equation}
where $1<i\leq 2$, and $^{C}D^{i,\kappa}_{0+}$ is the $\kappa$-fractional derivative of order $i$ in the sense
 of $\kappa$-Caputo operator, $f:[0,1]\times\mathbb{R}\rightarrow [0,\infty)$ is a continuous function.
 By employing the fixed point result of $\alpha-\phi$-Geraghty contractive type mappings, they obtained new results
 on the existence of positive solutions in $b$- metric spaces.\\

In \cite{Afshari4}, Afshari and Baleanu, investigated the Atangana–Baleanu fractional differential equations in the Caputo sense
\begin{equation}\label{eq03}
\begin{cases} (^{ABC}_{0}D^{i}u)(t)=f(t,u(t)),\ t \in (0,1),\\
u(0)=u_{0}
\end{cases}
       \end{equation}
where $^{ABC}_{0}D^{i}$ is the Atangana–Baleanu derivative in the Caputo sense of order $0<i\leq1$, and $f:[0,1]\times \mathcal{M}\rightarrow \mathbb{R}$ is continuous with $f(0,u(0))=0$ and $\mathcal{M}$ is a complete $b$-metric space.
The proof of main results are based upon some fixed point theorems for contractive mappings.

In \cite{Afshari9} Afshari et.al., considered the fractional boundary value problem
\begin{equation}\label{eq04}
\begin{cases} D^{i}u(t)=f(t,u(t)),\ t \in (0,1),\\
u(0)=u^{\prime}(0)=u(1)=u^{\prime}(1)=0,
\end{cases}
 \end{equation}

where $3<i\leq4$, $D^{i}$ is the classical Riemann–Liouville derivative, and $f:[0,1]\times \mathcal{X}\rightarrow\mathbb{R}$ is continuous, and $\mathcal{X}$ is a complete $b$-metric space. To obtain the existence of at least one solution, they used fixed point results of $\alpha-\psi$-Geraghty contractive type mappings.\\

Inspired and motiveted by the aforementioned works, in this paper, we are concerned with the following generalized fractional boundary value problem
 \begin{equation}\label{eq09}
  D^{\alpha,\varphi}u(t) + f(t,u(t)) = 0,\  t \in (0,1),
  \end{equation}
  \begin{equation}\label{eq010}
  u(0) = u^{\prime}(0) =0,\ u^{\prime}(1)=\beta u(\eta),
  \end{equation}
 where $2<\alpha\leq3$, $\beta\geq0$, $0<\eta\leq1$, and $f : [0, 1] \times {\mathbb{R}}\rightarrow {\mathbb{R}}$ is a continuous function,
$ D^{\alpha,\varphi}$ is the fractional derivative of order $\alpha$ in $\varphi$-Riemann-Liouville sense,
and $\varphi:[0,1]\rightarrow[0,1]$ is a strictly increasing function such that $ \varphi^{\prime}(x)\neq 0$ for
all $x\in[0,1]$. We use the technique of $\gamma-\psi$-GC type mappings and a modified version of contraction principle to investigate the existence of positive solutions for the fractional BVP \eqref{eq09}-\eqref{eq010}.\\

This paper is organized as follows. In section 2, we present some theorems and lemmas that will be used to prove our main results. In section 3, we investigate the existence and uniqueness of solutions for \eqref{eq09}-\eqref{eq010} in $b$-metric space. Our tools here are generalized Banach contraction principle and the technique of the $\gamma$–$\psi$-Geraghty contractive type mappings. Finally, we give two examples to illustrate our results in section 4.

 \section{Preliminaries}

\begin{definition}\label{defp1.1}\cite{Kilbas,Samko}
Let $\alpha>0$, $u:[a,b]\rightarrow\mathbb{R}$ be an integrable function and $\varphi\in C^{n}[a,b]$ an
increasing function such that $\varphi^{\prime}(t)\neq0$, for all $t\in[a, b]$.\\
The $\varphi$-Riemann–Liouville fractional integral of $u$ of order $\alpha$ is defined as follows:
\begin{equation*}
I_{a+}^{\alpha,\varphi}u(t)=\frac{1}{\Gamma(\alpha)}\int_{a}^{t}\varphi^{\prime}(s)
(\varphi(t)-\varphi(s))^{\alpha-1}u(s)ds,
\end{equation*}
and the $\varphi$-Riemann–Liouville fractional derivative of $u$ of order $\alpha$, with $n=[\alpha]+1$, is defined as follows:
\begin{eqnarray*}
D_{a+}^{\alpha,\varphi}u(t)&=&\Big(\frac{1}{\varphi^{\prime}(t)}\frac{d}{dt}\Big)^{n}
I_{a+}^{n-\alpha,\varphi}u(t)\\
&=&\frac{1}{\Gamma(n-\alpha)}\Big(\frac{1}{\varphi^{\prime}(t)}\frac{d}{dt}\Big)^{n}
\int_{a}^{t}\varphi^{\prime}(s)
(\varphi(t)-\varphi(s))^{n-\alpha-1}u(s)ds.
\end{eqnarray*}
\end{definition}

\begin{lemma}\label{lemp1.1} \cite{Kilbas}
Let $\alpha,\beta>0$ and $u:[a,b]\rightarrow\mathbb{R}$ be an integrable function. Then we have $D_{a+}^{\alpha,\varphi}I_{a+}^{\alpha,\varphi}u(t)=u(t)$ and
 $I_{a+}^{\alpha,\varphi}I_{a+}^{\beta,\varphi}u(t)=I_{a+}^{\alpha+\beta,\varphi}u(t)$.
\end{lemma}

\begin{lemma}\label{lemp1.2}
Let $\alpha>0$. Assume that $u\in C(a,b)\cap L(a,b)$, then the fractional differential equation
 $D_{a+}^{\alpha,\varphi}u(t)=0$ has a unique solution
 \[u(t)=c_{1}[\varphi(t)-\varphi(a)]^{\alpha-1}+c_{2}[\varphi(t)-\varphi(a)]^{\alpha-2}+
 \ldots+c_{n}[\varphi(t)-\varphi(a)]^{\alpha-n},\]
 Moreover, if $u,D_{a+}^{\alpha,\varphi}u\in C(a,b)\cap L(a,b)$, then
 \[I_{a+}^{\alpha,\varphi}D_{a+}^{\alpha,\varphi}u(t)=u(t)+c_{1}[\varphi(t)-\varphi(a)]^{\alpha-1}+c_{2}[\varphi(t)-\varphi(a)]^{\alpha-2}+
 \ldots+c_{n}[\varphi(t)-\varphi(a)]^{\alpha-n},\]
 where $c_{i}\in\mathbb{R}, i=1,2,\ldots,n.$
\end{lemma}

 \begin{definition}\label{defp1.2} \cite{Demmaa,Jovanovic,Czerwik}
Let $X$ be a nonempty set and let $r\geq1$ be a given real number.
A function $d : X \times X \rightarrow[0, \infty)$ is said to be a $b$-metric if and only if for all
$x, y, z\in X$ the following conditions are satisfied:
\begin{itemize}
\item [(1)] $d(x,y)=0$ if and only if $x=y$;
\item [(2)] $d(x,y)=d(y,x)$;
\item [(3)]$d(x,z)\leq r[d(x,y)+d(y,z)]$.
\end{itemize}
Then $(X,d,r)$ is called a $b$-metric space with constant $r$. Obviously, for $r = 1$ one obtains a metric on $X$.
 \end{definition}

\begin{theorem}\label{thmp1} \cite{Demmaa,Jovanovic}
Let $(X,d,r)$ be a complete $b$-metric space and let $A : X\rightarrow X$ be a map such that for some $\lambda$, $0<\lambda<\frac{1}{r}$,
\[d(Ax,Ay)\leq \lambda d(x,y)\]
holds for all $x,y\in X$. Then $A$ has a unique fixed point $z$, and for every $x_{0}\in X$, the sequence $\{{A^{n}x_{0}}\}$ converges to $z$.
\end{theorem}

Let $\Psi$ be the set of all increasing continuous functions $\psi : [0,\infty) \rightarrow [0,\infty)$ such that
$\psi(0)=0$, $\psi(\tau x)\leq \tau\psi(x)\leq \tau x$ for $\tau > 1$, and let $\Theta$ be the family of nondecreasing functions
 $\theta:[0,\infty)\rightarrow [0,\frac{1}{r^{2}})$ for some $r \geq1$.

 \begin{definition}\label{defp1.3} \cite{Afshari1}
 Let $(X,d,r)$ be a $b$-metric space. An operator $A:X\rightarrow X$ is called a generalized $\gamma$-$\psi$-Geraphty mapping whenever there exists $\gamma:X\times X\rightarrow [0,\infty)$ such that
 \[\gamma(x,y)\psi\big(r^{3}d(Ax,Ay)\big)\leq \theta \big(\psi(d(x,y))\big)\psi(d(x,y)),\]
 for all $x,y \in X$, where $\theta\in\Theta$ and $\psi\in\Psi$.
  \end{definition}

  \begin{definition}\label{defp1} \cite{Afshari1,Samet}
For $X\neq\emptyset$, let $A:X\rightarrow X$ and $\gamma:X\times X\rightarrow [0,\infty)$ be given mappings. Then $A$ is called $\gamma$-admissible if for $x,y \in X$, we have
\[\gamma(x,y)\geq1 \Rightarrow \gamma(Ax,Ay)\geq1.\]
  \end{definition}

    \begin{theorem}\label{thmp2} \cite{Afshari1}
    Let $(X,d,r)$ be a complete $b$-metric space, and let $A:X\rightarrow X$  be a
generalized $\gamma$-$\psi$-Geraghty mapping such that
\begin{itemize}
\item [(i)] $A$ is $\gamma$-admissible,
\item [(ii)] there exists $x_{0}\in X$ such that $\gamma(x_{0},Ax_{0})\geq1$,
\item [(iii)] if $\{x_{n}\}\subseteq X$ with $x_{n}\rightarrow x$ in $X$, and $\gamma(x_{n},x_{n+1})\geq1$, then  $\gamma(x_{n},x)\geq1$.
\end{itemize}
Then $A$ has a fixed point.
      \end{theorem}

For convenience, we denote
$$ \mu= (\alpha-1)\varphi^{\prime}(1)[\Phi(1)]^{\alpha-2}-\beta[\Phi(\eta)]^{\alpha-1},$$ $$\Lambda(t)=\frac{\Phi(s)}{\Phi(t)}$$ with $$\Phi(t)=\varphi(t)-\varphi(0).$$
 
\begin{lemma}\label{lem1.3}
Let $\mu\neq0$ and $2<\alpha\leq3$. Then, for any $ h \in C([0,1],\mathbb{R})$, the fractional boundary value problem
\begin{equation}\label{eq1}
D^{\alpha,\varphi}u(t) + h(t) = 0,\  t \in (0,1),
\end{equation}
\begin{equation}\label{eq2}
 u(0) = u^{\prime}(0) =0,\ u^{\prime}(1)=\beta u(\eta) ,
 \end{equation}
has an integral solution given by
$$ u(t) = \int_{0}^{1} \mathcal{ G}(t,s)\varphi^{\prime}(s)h(s) ds,$$
where
\begin{equation}\label{eq3}
\mathcal{G}(t,s)=\frac{1}{\mu\Gamma(\alpha)} \begin{cases} [\Phi(t)]^{\alpha-1}\big[(\alpha-1)\varphi^{\prime}(1)(\varphi(1)-\varphi(s))^{\alpha-2}
-\beta(\varphi(\eta)-\varphi(s))^{\alpha-1}\big]\\
-\mu(\varphi(t)-\varphi(s))^{\alpha-1}, & s \leq \min\{\eta,t\}, \\
 [\Phi(t)]^{\alpha-1}\big[(\alpha-1)\varphi^{\prime}(1)(\varphi(1)-\varphi(s))^{\alpha-2}
-\beta(\varphi(\eta)-\varphi(s))^{\alpha-1}\big], & t \leq s \leq \eta,\\
[\Phi(t)]^{\alpha-1}(\alpha-1)\varphi^{\prime}(1)(\varphi(1)-\varphi(s))^{\alpha-2}
-\mu(\varphi(t)-\varphi(s))^{\alpha-1},&\eta \leq s \leq t,\\
[\Phi(t)]^{\alpha-1}(\alpha-1)\varphi^{\prime}(1)(\varphi(1)-\varphi(s))^{\alpha-2}, &\max\{\eta,t\}\leq s.
\end{cases}
\end{equation}
\end{lemma}

\begin{proof}
From Lemma \ref{lemp1.2}, we may reduce \eqref{eq1}to an equivalent integral equation
\begin{eqnarray} \label{eq4}
u(t) & =&  -\frac{1}{\Gamma(\alpha)} \int_{0}^{t}\varphi^{\prime}(s)(\varphi(t)-\varphi(s))^{\alpha-1}h(s)ds \nonumber\\ &&+c_{1}(\varphi(t)-\varphi(0))^{\alpha-1}+c_{2}(\varphi(t)-\varphi(0))^{\alpha-2}+c_{3}(\varphi(t)-\varphi(0))^{\alpha-3},
\end{eqnarray}
where $c_{1}, c_{2}, c_{3}\in \mathbb{R}$ are arbitrary constants.\\
By using the boundary conditions $ u(0) = u^{\prime}(0) =0$, we get $c_{2}=c_{3}=0$. Then, Eq.\eqref{eq4} takes the following form
\begin{equation} \label{eq5}
u(t) =-\frac{1}{\Gamma(\alpha)} \int_{0}^{t}\varphi^{\prime}(s)(\varphi(t)-\varphi(s))^{\alpha-1}h(s)ds+c_{1}(\varphi(t)-\varphi(0))^{\alpha-1}.
\end{equation}
From $u^{\prime}(1)=\beta u(\eta)$, we have
\begin{multline*}
(\alpha-1)c_{1}\varphi^{\prime}(1)[\Phi(1)]^{\alpha-2}-\frac{\alpha-1}{\Gamma(\alpha)}
\int_{0}^{1}\varphi^{\prime}(s)\varphi^{\prime}(1)(\varphi(1)-\varphi(s))^{\alpha-2}h(s)ds=\beta c_{1}[\Phi(\eta)]^{\alpha-1}\\
-\frac{\beta}{\Gamma(\alpha)}
\int_{0}^{\eta}\varphi^{\prime}(s)(\varphi(\eta)-\varphi(s))^{\alpha-1}h(s)ds
\end{multline*}

So
\begin{equation*}
c_{1}=\frac{\alpha-1}{\mu\Gamma(\alpha)}
\int_{0}^{1}\varphi^{\prime}(s)\varphi^{\prime}(1)(\varphi(1)-\varphi(s))^{\alpha-2}h(s)ds
-\frac{\beta}{\mu\Gamma(\alpha)}\int_{0}^{\eta}\varphi^{\prime}(s)(\varphi(\eta)-\varphi(s))^{\alpha-1}h(s)ds
\end{equation*}
 Therefore, the solution of fractional boundary value problem \eqref{eq1}-\eqref{eq2} is
 \begin{eqnarray*}
u(t)& =& \frac{(\alpha-1)[\Phi(t)]^{\alpha-1}}{\mu\Gamma(\alpha)}
\int_{0}^{1}\varphi^{\prime}(s)\varphi^{\prime}(1)(\varphi(1)-\varphi(s))^{\alpha-2}h(s)ds\\
&&-\frac{\beta [\Phi(t)]^{\alpha-1}}{\mu\Gamma(\alpha)}\int_{0}^{\eta}\varphi^{\prime}(s)(\varphi(\eta)-\varphi(s))^{\alpha-1}h(s)ds \\
&&-\frac{1}{\Gamma(\alpha)}\int_{0}^{t}\varphi^{\prime}(s)(\varphi(t)-\varphi(s))^{\alpha-1}h(s)ds.
 \end{eqnarray*} 
For $t\leq\eta$, one has
\begin{eqnarray*}
u(t)& =& \frac{(\alpha-1)[\Phi(t)]^{\alpha-1}}{\mu\Gamma(\alpha)}
\int_{0}^{\eta}\varphi^{\prime}(s)\varphi^{\prime}(1)(\varphi(1)-\varphi(s))^{\alpha-2}h(s)ds\\
&&-\frac{\beta [\Phi(t)]^{\alpha-1}}{\mu\Gamma(\alpha)}\int_{0}^{\eta}\varphi^{\prime}(s)(\varphi(\eta)-\varphi(s))^{\alpha-1}h(s)ds \\
&&+\frac{(\alpha-1)[\Phi(t)]^{\alpha-1}}{\mu\Gamma(\alpha)}
\int_{\eta}^{1}\varphi^{\prime}(s)\varphi^{\prime}(1)(\varphi(1)-\varphi(s))^{\alpha-2}h(s)ds\\
&&-\frac{1}{\Gamma(\alpha)}\int_{0}^{t}\varphi^{\prime}(s)(\varphi(t)-\varphi(s))^{\alpha-1}h(s)ds\\
&=&\frac{(\alpha-1)[\Phi(t)]^{\alpha-1}}{\mu\Gamma(\alpha)}
\int_{0}^{t}\varphi^{\prime}(s)\varphi^{\prime}(1)(\varphi(1)-\varphi(s))^{\alpha-2}h(s)ds\\
&&+\frac{(\alpha-1)[\Phi(t)]^{\alpha-1}}{\mu\Gamma(\alpha)}
\int_{t}^{\eta}\varphi^{\prime}(s)\varphi^{\prime}(1)(\varphi(1)-\varphi(s))^{\alpha-2}h(s)ds\\
&&-\frac{\beta [\Phi(t)]^{\alpha-1}}{\mu\Gamma(\alpha)}\int_{0}^{t}\varphi^{\prime}(s)(\varphi(\eta)-\varphi(s))^{\alpha-1}h(s)ds \\
&&-\frac{\beta [\Phi(t)]^{\alpha-1}}{\mu\Gamma(\alpha)}\int_{t}^{\eta}\varphi^{\prime}(s)(\varphi(\eta)-\varphi(s))^{\alpha-1}h(s)ds\\ &&+\frac{(\alpha-1)[\Phi(t)]^{\alpha-1}}{\mu\Gamma(\alpha)}
\int_{\eta}^{1}\varphi^{\prime}(s)\varphi^{\prime}(1)(\varphi(1)-\varphi(s))^{\alpha-2}h(s)ds\\
&&-\frac{1}{\Gamma(\alpha)}\int_{0}^{t}\varphi^{\prime}(s)(\varphi(t)-\varphi(s))^{\alpha-1}h(s)ds
\end{eqnarray*}  
\begin{eqnarray*}&=&\frac{1}{\mu\Gamma(\alpha)}\int_{0}^{t} \Big([\Phi(t)]^{\alpha-1}\big[(\alpha-1)\varphi^{\prime}(1)(\varphi(1)-\varphi(s))^{\alpha-2}\\
&&-\beta(\varphi(\eta)-\varphi(s))^{\alpha-1}\big]
-\mu(\varphi(t)-\varphi(s))^{\alpha-1}\Big) \varphi^{\prime}(s)h(s)ds\\
&&+\int_{t}^{\eta}\frac{[\Phi(t)]^{\alpha-1}\big[(\alpha-1)\varphi^{\prime}(1)(\varphi(1)-\varphi(s))^{\alpha-2}
-\beta(\varphi(\eta)-\varphi(s))^{\alpha-1}\big]}{\mu\Gamma(\alpha)}\varphi^{\prime}(s)h(s)ds\\
&&+\int_{\eta}^{1}\frac{[\Phi(t)]^{\alpha-1}(\alpha-1)\varphi^{\prime}(1)(\varphi(1)-\varphi(s))^{\alpha-2}}
{\mu\Gamma(\alpha)}\varphi^{\prime}(s)h(s)ds\\
&=&\int_{0}^{1}\mathcal{G}(t,s)\varphi^{\prime}(s)h(s)ds.
\end{eqnarray*}
For $\eta\leq t$, one has
\begin{eqnarray*}
u(t)& =& \frac{(\alpha-1)[\Phi(t)]^{\alpha-1}}{\mu\Gamma(\alpha)}
\int_{0}^{t}\varphi^{\prime}(s)\varphi^{\prime}(1)(\varphi(1)-\varphi(s))^{\alpha-2}h(s)ds\\
&&+\frac{(\alpha-1)[\Phi(t)]^{\alpha-1}}{\mu\Gamma(\alpha)}
\int_{t}^{1}\varphi^{\prime}(s)\varphi^{\prime}(1)(\varphi(1)-\varphi(s))^{\alpha-2}h(s)ds\\
&&-\frac{\beta [\Phi(t)]^{\alpha-1}}{\mu\Gamma(\alpha)}\int_{0}^{\eta}\varphi^{\prime}(s)(\varphi(\eta)-\varphi(s))^{\alpha-1}h(s)ds \\
&&-\frac{1}{\Gamma(\alpha)}\int_{0}^{t}\varphi^{\prime}(s)(\varphi(t)-\varphi(s))^{\alpha-1}h(s)ds\\
& =&\frac{(\alpha-1)[\Phi(t)]^{\alpha-1}}{\mu\Gamma(\alpha)}
\int_{0}^{\eta}\varphi^{\prime}(s)\varphi^{\prime}(1)(\varphi(1)-\varphi(s))^{\alpha-2}h(s)ds\\
&&+\frac{(\alpha-1)[\Phi(t)]^{\alpha-1}}{\mu\Gamma(\alpha)}
\int_{\eta}^{t}\varphi^{\prime}(s)\varphi^{\prime}(1)(\varphi(1)-\varphi(s))^{\alpha-2}h(s)ds\\
&&+\frac{(\alpha-1)[\Phi(t)]^{\alpha-1}}{\mu\Gamma(\alpha)}
\int_{t}^{1}\varphi^{\prime}(s)\varphi^{\prime}(1)(\varphi(1)-\varphi(s))^{\alpha-2}h(s)ds
\end{eqnarray*}
\begin{eqnarray*}
&&-\frac{\beta [\Phi(t)]^{\alpha-1}}{\mu\Gamma(\alpha)}\int_{0}^{\eta}\varphi^{\prime}(s)(\varphi(\eta)-\varphi(s))^{\alpha-1}h(s)ds \\
&&-\frac{1}{\Gamma(\alpha)}\int_{0}^{\eta}\varphi^{\prime}(s)(\varphi(t)-\varphi(s))^{\alpha-1}h(s)ds\\
&&-\frac{1}{\Gamma(\alpha)}\int_{\eta}^{t}\varphi^{\prime}(s)(\varphi(t)-\varphi(s))^{\alpha-1}h(s)ds\\
& =&\frac{1}{\mu\Gamma(\alpha)}\int_{0}^{\eta} \Big([\Phi(t)]^{\alpha-1}\big[(\alpha-1)\varphi^{\prime}(1)(\varphi(1)-\varphi(s))^{\alpha-2}\\
&&-\beta(\varphi(\eta)-\varphi(s))^{\alpha-1}\big]
-\mu(\varphi(t)-\varphi(s))^{\alpha-1}\Big) \varphi^{\prime}(s)h(s)ds\\
&&+\int_{\eta}^{t}\frac{[\Phi(t)]^{\alpha-1}(\alpha-1)\varphi^{\prime}(1)(\varphi(1)-\varphi(s))^{\alpha-2}
-\mu(\varphi(t)-\varphi(s))^{\alpha-1}}{\mu\Gamma(\alpha)}\varphi^{\prime}(s)h(s)ds\\
&&+\int_{t}^{1}\frac{[\Phi(t)]^{\alpha-1}(\alpha-1)\varphi^{\prime}(1)(\varphi(1)-\varphi(s))^{\alpha-2}}
{\mu\Gamma(\alpha)}\varphi^{\prime}(s)h(s)ds\\
&=&\int_{0}^{1}\mathcal{G}(t,s)\varphi^{\prime}(s)h(s)ds.
\end{eqnarray*}
\end{proof}
Now we derive some properties of Green's function $\mathcal{G}(t,s)$.
\begin{lemma}\label{lem1.4}
Let $2<\alpha\leq3$ and $0\leq\beta<(\alpha-1)\varphi^{\prime}(1)\frac{[\Phi(1)]^{\alpha-2}}{[\Phi(\eta)]^{\alpha-1}}$. Then the Green's function $\mathcal{G}(t,s)$ defined by \eqref{eq3} satisfies the following properties:
\begin{itemize}
\item[(i)]  $\mathcal{G}(t,s)$ is continuous on $[0,1]\times[0,1],$
\item[(ii)] $\mathcal{G}(t,s)>0$, for all $t,s \in(0,1),$
\item[(iii)] For $s\in(0,1)$, we have
$$\max_{t\in[0,1]}\mathcal{G}(t,s)\leq
\frac{(\alpha-1)\varphi^{\prime}(1)(\varphi(1)-\varphi(s))^{\alpha-2}}
{\mu\Gamma(\alpha)}.$$
\end{itemize}
\end{lemma}
\begin{proof}
\rm{(i)} It is easy to check that $\mathcal{G}(t,s)$ is continuous on $[0,1]\times[0,1]$.\\
\rm{(ii)} For $s \leq \min\{\eta,t\}$, by \eqref{eq3}, and observing that $\Lambda$ is decreasing function, one has
\begin{eqnarray*}
\mu\Gamma(\alpha)\mathcal{G}(t,s)&=&[\Phi(t)]^{\alpha-1}\big[(\alpha-1)\varphi^{\prime}(1)(\varphi(1)-\varphi(s))^{\alpha-2}
-\beta(\varphi(\eta)-\varphi(s))^{\alpha-1}\big]\\
&&-\mu(\varphi(t)-\varphi(s))^{\alpha-1}\\
&=& [\Phi(t)]^{\alpha-1}\Bigg[(\alpha-1)\varphi^{\prime}(1)(\varphi(1)-\varphi(s))^{\alpha-2}
-\beta(\varphi(\eta)-\varphi(s))^{\alpha-1}\\
&&-\mu\frac{(\varphi(t)-\varphi(s))^{\alpha-1}}{(\varphi(t)-\varphi(0))^{\alpha-1}}\Bigg]\\
&=&[\Phi(t)]^{\alpha-1}\Big[(\alpha-1)\varphi^{\prime}(1)[\Phi(1)]^{\alpha-2}(1-\Lambda(1))^{\alpha-2}
-\beta[\Phi(\eta)]^{\alpha-1}(1-\Lambda(\eta))^{\alpha-1}\\
&&-\mu(1-\Lambda(t))^{\alpha-1}\Big]\\
&=&[\Phi(t)]^{\alpha-1}\Big[(\alpha-1)\varphi^{\prime}(1)[\Phi(1)]^{\alpha-2}(1-\Lambda(1))^{\alpha-2}
-\beta[\Phi(\eta)]^{\alpha-1}(1-\Lambda(\eta))^{\alpha-1}\\
&&-(\alpha-1)\varphi^{\prime}(1)[\Phi(1)]^{\alpha-2}(1-\Lambda(t))^{\alpha-1}+\beta[\Phi(\eta)]^{\alpha-1}(1-\Lambda(t))^{\alpha-1}\Big].\\
\end{eqnarray*}
Next, we consider two cases:\\
\textbf{Case 1:} If $\eta\leq t$, then
\begin{eqnarray*}
\mu\Gamma(\alpha)\mathcal{G}(t,s)&=&[\Phi(t)]^{\alpha-1}\Big[(\alpha-1)\varphi^{\prime}(1)[\Phi(1)]^{\alpha-2}
\big[(1-\Lambda(1))^{\alpha-2}-(1-\Lambda(t))^{\alpha-1}\big]\\
&&+\beta[\Phi(\eta)]^{\alpha-1}
\big[(1-\Lambda(t))^{\alpha-1}-(1-\Lambda(\eta))^{\alpha-1}\Big]\\
&>&[\Phi(t)]^{\alpha-1}\Big[(\alpha-1)\varphi^{\prime}(1)[\Phi(1)]^{\alpha-2}
\big[(1-\Lambda(t))^{\alpha-2}-(1-\Lambda(t))^{\alpha-1}\big]\\
&&+\beta[\Phi(\eta)]^{\alpha-1}
\big[(1-\Lambda(t))^{\alpha-1}-(1-\Lambda(\eta))^{\alpha-1}\Big]\\
&=&[\Phi(t)]^{\alpha-1}\Big[(\alpha-1)\varphi^{\prime}(1)[\Phi(1)]^{\alpha-2}\Lambda(t)(1-\Lambda(t))^{\alpha-2}\\
&&+\beta[\Phi(\eta)]^{\alpha-1}
\big[(1-\Lambda(t))^{\alpha-1}-(1-\Lambda(\eta))^{\alpha-1}\Big]\\
&>&0.
\end{eqnarray*}
\textbf{Case 2:} If $t \leq \eta$, then
\begin{eqnarray*}
\mu\Gamma(\alpha)\mathcal{G}(t,s)&=&[\Phi(t)]^{\alpha-1}\Big[(\alpha-1)\varphi^{\prime}(1)[\Phi(1)]^{\alpha-2}(1-\Lambda(1))^{\alpha-2}
-\beta[\Phi(\eta)]^{\alpha-1}(1-\Lambda(\eta))^{\alpha-1}\\
&&-\mu(1-\Lambda(t))^{\alpha-1}\Big]\\
&>&[\Phi(t)]^{\alpha-1}\Big[(\alpha-1)\varphi^{\prime}(1)[\Phi(1)]^{\alpha-2}(1-\Lambda(1))^{\alpha-2}
-\beta[\Phi(\eta)]^{\alpha-1}(1-\Lambda(\eta))^{\alpha-1}\\
&&-\mu(1-\Lambda(\eta))^{\alpha-1}\Big]\\
&=&(\alpha-1)[\Phi(t)]^{\alpha-1}\varphi^{\prime}(1)[\Phi(1)]^{\alpha-2}
\big[(1-\Lambda(1))^{\alpha-2}-(1-\Lambda(\eta))^{\alpha-1}\Big]\\
&>&(\alpha-1)[\Phi(t)]^{\alpha-1}\varphi^{\prime}(1)[\Phi(1)]^{\alpha-2}
\big[(1-\Lambda(\eta))^{\alpha-2}-(1-\Lambda(\eta))^{\alpha-1}\Big]\\
&=&(\alpha-1)[\Phi(t)]^{\alpha-1}\varphi^{\prime}(1)[\Phi(1)]^{\alpha-2}\Lambda(\eta)(1-\Lambda(\eta))^{\alpha-2}\\
&>&0.
\end{eqnarray*}
For $t\leq s\leq \eta$, one has
\begin{eqnarray*}
\mu\Gamma(\alpha)\mathcal{G}(t,s)&=&[\Phi(t)]^{\alpha-1}\big[(\alpha-1)\varphi^{\prime}(1)(\varphi(1)-\varphi(s))^{\alpha-2}
-\beta(\varphi(\eta)-\varphi(s))^{\alpha-1}\big]\\
&=&[\Phi(t)]^{\alpha-1}\Big[(\alpha-1)\varphi^{\prime}(1)[\Phi(1)]^{\alpha-2}(1-\Lambda(1))^{\alpha-2}
-\beta[\Phi(\eta)]^{\alpha-1}(1-\Lambda(\eta))^{\alpha-1}\Big]\\
&>&[\Phi(t)]^{\alpha-1}\Big[\beta[\Phi(\eta)]^{\alpha-1}(1-\Lambda(1))^{\alpha-2}
-\beta[\Phi(\eta)]^{\alpha-1}(1-\Lambda(\eta))^{\alpha-1}\Big]\\
&>&\beta[\Phi(t)\Phi(\eta)]^{\alpha-1}\Big[(1-\Lambda(\eta))^{\alpha-2}
-(1-\Lambda(\eta))^{\alpha-1}\Big]\\
&=&\beta[\Phi(t)\Phi(\eta)]^{\alpha-1}(1-\Lambda(\eta))^{\alpha-2}\Lambda(\eta)\\
&>&0.
\end{eqnarray*}
For $\eta\leq s\leq t$, one has
\begin{eqnarray*}
\mu\Gamma(\alpha)\mathcal{G}(t,s)&=&[\Phi(t)]^{\alpha-1}(\alpha-1)\varphi^{\prime}(1)(\varphi(1)-\varphi(s))^{\alpha-2}
-\mu(\varphi(t)-\varphi(s))^{\alpha-1}\\
&=&(\alpha-1)[\Phi(t)]^{\alpha-1}\varphi^{\prime}(1)[\Phi(1)]^{\alpha-2}(1-\Lambda(1))^{\alpha-2}
-\mu[\Phi(t)]^{\alpha-1}(1-\Lambda(t))^{\alpha-1}\\
&=&(\alpha-1)[\Phi(t)]^{\alpha-1}\varphi^{\prime}(1)[\Phi(1)]^{\alpha-2}(1-\Lambda(1))^{\alpha-2}\\
&&-(\alpha-1)[\Phi(t)]^{\alpha-1}\varphi^{\prime}(1)[\Phi(1)]^{\alpha-2}(1-\Lambda(t))^{\alpha-1}
+\beta[\Phi(\eta)\Phi(t)]^{\alpha-1}(1-\Lambda(t))^{\alpha-1}\\
&>&(\alpha-1)[\Phi(t)]^{\alpha-1}\varphi^{\prime}(1)[\Phi(1)]^{\alpha-2}\Big[(1-\Lambda(1))^{\alpha-2}
-(1-\Lambda(t))^{\alpha-1}\Big]\\
&>&(\alpha-1)[\Phi(t)]^{\alpha-1}\varphi^{\prime}(1)[\Phi(1)]^{\alpha-2}\Big[(1-\Lambda(t))^{\alpha-2}
-(1-\Lambda(t))^{\alpha-1}\Big]\\
&=&(\alpha-1)[\Phi(t)]^{\alpha-1}\varphi^{\prime}(1)[\Phi(1)]^{\alpha-2}(1-\Lambda(t))^{\alpha-2} \Lambda(t)\\
&>&0.
\end{eqnarray*}
Clearly, for $\max\{\eta,t\}\leq s$, $\mathcal{G}(t,s)>0$.\\
Therefore, $\mathcal{G}(t,s)>0$ for $t,s\in(0,1)$.\\
\rm{(iii)} Since $\varphi$ is a strictly increasing function, then by \eqref{eq3}, it is easily seen that
 \begin{eqnarray*}
\mu\Gamma(\alpha)\max_{t\in[0,1]}\mathcal{G}(t,s)&\leq& (\alpha-1)[\Phi(t)]^{\alpha-1}
\varphi^{\prime}(1)(\varphi(1)-\varphi(s))^{\alpha-2}\\
&\leq&(\alpha-1)\varphi^{\prime}(1)(\varphi(1)-\varphi(s))^{\alpha-2},\ \text{for all}\ s\in(0,1).\\
 \end{eqnarray*}
Therefore
$$\max_{t\in[0,1]}\mathcal{G}(t,s)\leq\frac{(\alpha-1)\varphi^{\prime}(1)
(\varphi(1)-\varphi(s))^{\alpha-2}}{\mu\Gamma(\alpha)},\ \text{for all}\ s\in(0,1).$$
The proof is completed.
\end{proof}
\section{Main results}
By $X = \mathcal{C}([0,1],\mathbb{R})$ we denote the set of continuous functions. Let $d : X\times X\rightarrow[0,\infty)$ be given by
\[d(x,y)=\|(x-y)^{2}\|_{\infty}=\sup_{t\in[0,1]}(x(t)-y(t))^{2}.\]
Then, $(X, d, r)$ is a complete $b$-metric space with $r = 2$.\\
Let $\mathcal{A}: X \rightarrow X$ be the operator defined as
\begin{equation}\label{eq3.1}
\mathcal{A}u(t)= \int_{0}^{1} \mathcal{ G}(t,s)\varphi^{\prime}(s)f(s,u(s))ds,\ t\in[0,1].
\end{equation}
Clearly, $u$ is a solution of \eqref{eq09}-\eqref{eq010} if and only if $u$ is a fixed point of operator $\mathcal{A}$.\\
We now list suitable conditions on the nonlinearity function $f(t,u)$.
 \begin{itemize}
\item[(H1)] $f\in \mathcal{C}([0,1]\times \mathbb{R},\mathbb{R})$;
\item[(H2)] There exists real-valued function not identically zero $g\in \mathcal{C}([0,1], \mathbb{R}_{+})$ such that for all
$t\in[0,1]$ and $u, v\in \mathbb{R}$, we have
    \[|f(t,u)-f(t,v)|\leq g(t)|u-v|.\]
\end{itemize}
Our first result on the existence and uniqueness of solutions is based on the Banach contraction principle  in a $b$-metric space.
\begin{theorem} \label{thmm1}
Assume that \rm{(H1)}-\rm{(H2)} hold. Then the fractional boundary value problem \eqref{eq09}-\eqref{eq010} has a unique solution on $[0,1]$ provided that
\begin{equation}\label{eq3.2}
\|g\|_{\infty}<\frac{\mu\Gamma(\alpha)}{\sqrt{2}\big[\Phi(1)\big]^{\alpha-1}\varphi^{\prime}(1)}.
\end{equation}
\end{theorem}

\begin{proof}
By Lemmas \ref{lem1.3}-\ref{lem1.4}, for $u,v\in X$, and $t\in [0,1]$, we have
 \begin{eqnarray*}
 \big|(\mathcal{A}u)(t)-(\mathcal{A}v)(t)\big|^{2}&=&\Bigg|\int_{0}^{1} \mathcal{ G}(t,s)\varphi^{\prime}(s)\Big(f(s,u(s))-f(s,v(s))\Big)ds\Bigg|^{2} \\
  &\leq&\Bigg[\int_{0}^{1} \mathcal{ G}(t,s)\varphi^{\prime}(s)\Big|f(s,u(s))-f(s,v(s))\Big|ds\Bigg]^{2} \\
 &\leq&\Bigg[\int_{0}^{1} \mathcal{ G}(t,s)\varphi^{\prime}(s)g(s)\big|u(s)-v(s)\big|ds\Bigg]^{2} \\
&\leq&\Bigg[\int_{0}^{1} \mathcal{ G}(t,s)\varphi^{\prime}(s)\|g\|_{\infty}\sqrt{d(u,v)}ds\Bigg]^{2} \\
&\leq&\Bigg(\frac{(\alpha-1)\|g\|_{\infty}\varphi^{\prime}(1)}{\mu\Gamma(\alpha)}\Bigg)^{2}
\Bigg[\int_{0}^{1} \big(\varphi(1)-\varphi(s)\big)^{\alpha-2}\varphi^{\prime}(s)ds\Bigg]^{2}d(u,v)\\
&=&\Bigg(\frac{(\alpha-1)\|g\|_{\infty}\varphi^{\prime}(1)}{\mu\Gamma(\alpha)}\Bigg)^{2}
\Bigg[ \frac{1}{\alpha-1}(\varphi(1)-\varphi(0))^{\alpha-1} \Bigg]^{2}d(u,v)\\
&=&\Bigg(\frac{\|g\|_{\infty}\varphi^{\prime}(1)\big[\Phi(1)\big]^{\alpha-1}}{\mu\Gamma(\alpha)}\Bigg)^{2}d(u,v)
   \end{eqnarray*}
By denoting $\lambda=\Big(\frac{\|g\|_{\infty}\varphi^{\prime}(1)[\Phi(1)]^{\alpha-1}}{\mu\Gamma(\alpha)}\Big)^{2},$ we get
\[d(\mathcal{A}u,\mathcal{A}v)\leq
\lambda d(u,v).\]
From \eqref{eq3.2} and by means of Theorem \ref{thmp1}, we claim that the operator $\mathcal{A}$ has a unique fixed point.
The proof is completed.
\end{proof}
The second result concerns the existence of positive solutions for the $\varphi$-Riemann-Liouville fractional differential equation
\eqref{eq09}-\eqref{eq010} via the $\gamma$–$\psi$-Geraghty
contractive type mappings. To begin with, we make the following assumption:

 \begin{itemize}
\item[(H3)] $f \in C([0,1]\times \mathbb{R^{+}}, \mathbb{R}^{+})$.
\end{itemize}
\begin{theorem}\label{thmm2}
Assume that \rm{(H3)} holds. If the following assumptions are satisfied:
\begin{itemize}
  \item [(i)] There exists $f:[0,1]\times \mathbb{R^{+}}\rightarrow \mathbb{R^{+}}$ such that
\begin{eqnarray*}
\big|f(s,u(s))-f(s,v(s))\big|&\leq&\frac{1}{2\sqrt{2}}\frac{\mu\Gamma(\alpha)}{\varphi^{\prime}(1)\big[\Phi(1)\big]^{\alpha-1}}\\
&&\times\sqrt{\psi\big(\|(u-v)^{2}\|_{\infty}\big)\theta\big(\psi\big(\|(u-v)^{2}\|_{\infty}\big)\big)},\ s\in[0,1],
\end{eqnarray*}
where $\psi\in\Psi$ and $\theta\in\Theta$;
  \item [(ii)] There exists $u_{0}\in \mathcal{C}([0,1])$ and $\tau:\mathbb{R}^{2}\rightarrow \mathbb{R}$ such that
\[\tau\big(u_{0}(t),\mathcal{A}u_{0}(t)\big)\geq0,\ t\in[0,1];\]
  \item [(iii)] For $t\in[0,1]$ and $u,v\in \mathcal{C}([0,1])$, $\tau\big(u(t),v(t)\big)\geq0$ implies
\[\tau\big(\mathcal{A}u(t),\mathcal{A}v(t)\big)\geq0;\]
  \item [(iv)] If $\{u_{n}\}\subseteq \mathcal{C}([0,1])$ with $u_{n}\rightarrow u$ in $\mathcal{C}([0,1])$, and $\tau(u_{n},u_{n+1})\geq0$, then $\tau(u_{n},u)\geq0$,
\end{itemize}
then the fractional boundary value problem \eqref{eq09}-\eqref{eq010} has at least one positive solution.
\end{theorem}
\begin{proof}
By Lemma \ref{lem1.3}, $u\in \mathcal{C}([0,1])$ is a solution of \eqref{eq09}-\eqref{eq010} if and only if $u$ is a solution of the integral equation
\[u(t)=\int_{0}^{1} \mathcal{ G}(t,s)\varphi^{\prime}(s)f(s,u(s))ds,\ t\in [0,1].\]
Clearly, the fixed points of $\mathcal{A}$ coincide with the solution of fractional boundary value problem \eqref{eq09}-\eqref{eq010}.\\
Let $u,v\in \mathcal{C}([0,1])$ be such that $\tau\big(u(t),v(t)\big)\geq0$ for $t\in[0,1]$.\\
Using \rm{(i)} and Lemma \ref{lem1.4}, we get
\begin{eqnarray*}
\big|(\mathcal{A}u)(t)-(\mathcal{A}v)(t)\big|^{2}&=&\Bigg|\int_{0}^{1} \mathcal{ G}(t,s)\varphi^{\prime}(s)\Big(f(s,u(s))-f(s,v(s))\Big)ds\Bigg|^{2} \\
  &\leq&\Bigg[\int_{0}^{1} \mathcal{ G}(t,s)\varphi^{\prime}(s)\Big|f(s,u(s))-f(s,v(s))\Big|ds\Bigg]^{2} \\
  &\leq&\Bigg[\int_{0}^{1} \mathcal{ G}(t,s)\varphi^{\prime}(s)\frac{1}{2\sqrt{2}}\frac{\mu\Gamma(\alpha)}{\varphi^{\prime}(1)\big[\Phi(1)\big]^{\alpha-1}}\\
&&\times\sqrt{\psi\big(\|(u-v)^{2}\|_{\infty}\big)\theta\big(\psi\big(\|(u-v)^{2}\|_{\infty}\big)\big)}ds\Bigg]^{2}\\
  &\leq& \frac{1}{8}\psi\big(\|(u-v)^{2}\|_{\infty}\big)\theta\big(\psi\big(\|(u-v)^{2}\|_{\infty}\big)\big).
\end{eqnarray*}
Hence
\[\|(\mathcal{A}u-\mathcal{A}v)^{2}\|_{\infty}\leq \frac{1}{8}\psi\big(\|(u-v)^{2}\|_{\infty}\big)\theta\big(\psi\big(\|(u-v)^{2}\|_{\infty}\big)\big).\]
Let $\gamma:\mathcal{C}([0,1])\times \mathcal{C}([0,1])\rightarrow \mathbb{R^{+}}$ be defined by
\begin{equation*}
\gamma(u,v) = \begin{cases}1 ,& \tau\big(u(t),v(t)\big)\geq0,\ t\in[0,1],\\
0,& \text{otherwise}.
\end{cases}
\end{equation*}
So for $u,v\in \mathcal{C}([0,1])$ with $\tau\big(u(t),v(t)\big)\geq0$, $t\in[0,1]$, we have
\begin{eqnarray*}
\gamma(u,v)\psi\big(8d(\mathcal{A}u,\mathcal{A}v)\big)&\leq& 8 d(\mathcal{A}u,\mathcal{A}v)\\
&\leq&\theta\big(\psi(d(u,v))\big)\psi(d(u,v)),\ \theta\in\Theta.
\end{eqnarray*}
This proves $\mathcal{A}$ is a $\gamma-\psi$-Geraphty mapping.\\
Now, from \rm{(iii)}, we get for $u,v\in \mathcal{C}([0,1])$,
\begin{eqnarray*}
\gamma(u,v)\geq 1 &\Rightarrow & \tau\big(u(t),v(t)\big)\geq0\\
&\Rightarrow & \tau\big(\mathcal{A}u(t),\mathcal{A}v(t)\big)\geq0\\
&\Rightarrow & \gamma\big(\mathcal{A}u,\mathcal{A}v\big)\geq 1.
\end{eqnarray*}
Thus, $\mathcal{A}$ is $\gamma$-admissible.\\
By \rm{(ii)}, there exists $u_{0}\in \mathcal{C}([0,1])$ such that $\gamma\big(u_{0},\mathcal{A}u_{0}\big)\geq 1$. Utilizing \rm{(iv)} and Theorem \ref{thmp2}, there exists $u^{\star}\in \mathcal{C}([0,1])$ such that $u^{\star}=\mathcal{A}u^{\star}$. We  have proved that $u^{\star}$ is a solution of the problem \eqref{eq09}-\eqref{eq010}.
\end{proof}

 \section{Two illustrative examples}
\begin{exmp}
Let $\tau(x,y)=xy$, $\psi(t)=t$, $\theta(t)=\frac{1+t^{2}}{6+4t^{2}}$ for $t\geq0$.
We consider the fractional boundary value problem
\begin{equation}\label{eq4.1}
       \begin{cases}D^{\frac{5}{2}, \sin(\frac{\pi}{4}t)}u(t) + f(t,u(t)) = 0,\  t \in (0,1),\\
    u(0) = u^{\prime}(0) =0,\ u^{\prime}(1)=2 u\big(\frac{1}{2}\big),
       \end{cases}
       \end{equation}
      where $\alpha=\frac{5}{2}$, $\varphi(t)=\sin(\frac{\pi}{4}t)$, $\beta=2$ and $\eta=\frac{1}{2}$.\\
So
 \begin{eqnarray*}
(\alpha-1)\varphi^{\prime}(1)\frac{[\Phi(1)]^{\alpha-2}}{[\Phi(\eta)]^{\alpha-1}}&=&\frac{3\pi}{8}\sqrt{\frac{2\sqrt{2}}{(2-\sqrt{2})\sqrt{2-\sqrt{2}}}}\\
&\thickapprox&2.95903\\
&>&\beta=2,
 \end{eqnarray*}
 \begin{eqnarray*}
\mu&=& (\alpha-1)\varphi^{\prime}(1)[\Phi(1)]^{\alpha-2}-\beta[\Phi(\eta)]^{\alpha-1}\\
&=&\frac{3\pi}{8\times2^{\frac{3}{4}}}-\frac{(2-\sqrt{2})^{\frac{3}{4}}}{\sqrt{2}}\\
&\thickapprox&0.22703.
 \end{eqnarray*}
Considering
 \begin{eqnarray*}
f(t,u)&=&\frac{1}{16\sqrt{2}}\frac{\mu\Gamma(\alpha)}{\varphi^{\prime}(1)\big[\Phi(1)\big]^{\alpha-1}}u\\
&=&\frac{3}{8\times2^{1/4}\sqrt{\pi}}\Bigg[\frac{3\pi}{8\times2^{\frac{3}{4}}}
-\frac{(2-\sqrt{2})^{\frac{3}{4}}}{\sqrt{2}}\Bigg]u,\ \text{for}\ t\in[0,1].
 \end{eqnarray*}
We get
 \begin{eqnarray*}
|f(s,u(s))-f(s,v(s))|&=&\frac{1}{16\sqrt{2}}\frac{\mu\Gamma(\alpha)}{\varphi^{\prime}(1)\big[\Phi(1)\big]^{\alpha-1}}
|u(s)-v(s)|\\
&\leq&\frac{1}{2\sqrt{2}}\frac{\mu\Gamma(\alpha)}{\varphi^{\prime}(1)\big[\Phi(1)\big]^{\alpha-1}}\frac{\sqrt{d(u,v)}}{8}\\
&\leq&\frac{1}{2\sqrt{2}}\frac{\mu\Gamma(\alpha)}{\varphi^{\prime}(1)\big[\Phi(1)\big]^{\alpha-1}}\frac{\sqrt{d(u,v)}}{6}\\
&\leq&\frac{1}{2\sqrt{2}}\frac{\mu\Gamma(\alpha)}{\varphi^{\prime}(1)\big[\Phi(1)\big]^{\alpha-1}}\sqrt{d(u,v)}\theta(d(u,v)) \\
&\leq&\frac{1}{2\sqrt{2}}\frac{\mu\Gamma(\alpha)}{\varphi^{\prime}(1)\big[\Phi(1)\big]^{\alpha-1}}\sqrt{d(u,v)\theta(d(u,v))} \\
&=&\frac{1}{2\sqrt{2}}\frac{\mu\Gamma(\alpha)}{\varphi^{\prime}(1)\big[\Phi(1)\big]^{\alpha-1}}\\
&\times&\sqrt{\psi\big(\|(u-v)^{2}\|_{\infty}\big)\theta\big(\psi\big(\|(u-v)^{2}\|_{\infty}\big)\big)},\ s\in[0,1].
\end{eqnarray*}
So, assumption $\rm{(i)}$ from Theorem \ref{thmm2} is satisfied. It is obvious that assumptions $\rm{(ii)}-\rm{(iv)}$ in Theorem \ref{thmm2} hold. Therefore problem \eqref{eq4.1} has at least one positive solution.
\end{exmp}

\begin{exmp}
As a second example, we consider the following fractional boundary value problem
\begin{equation}\label{eq4.2}
       \begin{cases}D^{\frac{5}{2}, \frac{1}{2}\sqrt{1+t}}u(t) + f(t,u(t)) = 0,\  t \in (0,1),\\
    u(0) = u^{\prime}(0) =0,\ u^{\prime}(1)=4 u\big(\frac{1}{3}\big),
       \end{cases}
       \end{equation}
where  $ f(t,u) = \frac{1}{10}\tan\big(\frac{\pi}{3}t\big)\cos^{2}(u)-\frac{1}{3}e^{\frac{t}{2}}\frac{|u|}{1+|u|}$,
$\alpha=\frac{5}{2}$, $\varphi(t)=\frac{1}{2}\sqrt{1+t}$, $\beta=4$ and $\eta=\frac{1}{3}$.\\
By a direct computation, we get
 \begin{eqnarray*}
(\alpha-1)\varphi^{\prime}(1)\frac{[\Phi(1)]^{\alpha-2}}{[\Phi(\eta)]^{\alpha-1}}&=& \frac{9}{4}\sqrt{\frac{3}{2}}
\frac{\sqrt{\sqrt{2}-1}}{\big(2\sqrt{3}-3\big)\sqrt{2\sqrt{3}-3}}\\
&\thickapprox&5.60946\\
&>&\beta=4,
 \end{eqnarray*}
 \begin{eqnarray*}
\mu&=& (\alpha-1)\varphi^{\prime}(1)[\Phi(1)]^{\alpha-2}-\beta[\Phi(\eta)]^{\alpha-1}\\
&=&\frac{3}{16}\sqrt{\sqrt{2}-1}-\frac{1}{3}\sqrt{\frac{2}{3}}\big(2\sqrt{3}-3\big)\sqrt{2\sqrt{3}-3}\\
&\thickapprox&0.0346236.
 \end{eqnarray*}
Furthermore, by choosing $g_{1}(t)=\frac{1}{10}\tan\big(\frac{\pi}{3}t\big),
\ g_{2}(t)=\frac{1}{3}e^{\frac{t}{2}},$ and $g(t)=2g_{1}(t)+g_{2}(t)$,
we find that $\|g\|_{\infty}=g(1)=\frac{\sqrt{3}}{5}+\frac{\sqrt{e}}{3}\thickapprox0.895984$.
Also, for every $u_{1}, u_{2}, v_{1}, v_{1}\in\mathbb{R}$ and $t\in[0,1]$, we get
 \begin{eqnarray*}
 \big|f(t,u)-f(t,v)\big|&=&\Bigg|g_{1}(t)\big(\cos^{2}(u)-\cos^{2}(v)\big)+g_{2}(t)
\bigg(\frac{|v|-|u|}{(1+|u|)(1+|v|)}\bigg)\Bigg| \\
  &\leq&2 g_{1}(t) \big|\cos(u)-\cos(v) \big|+g_{2}(t)|u-v|\\
 &\leq&2 g_{1}(t) \big|u-v\big|+g_{2}(t)|u-v|\\
 &=&g(t)|u-v|.
   \end{eqnarray*}
So hypothesis $\rm{(H2)}$ is satisfied. Moreover, a simple computation gives
 \begin{eqnarray*}
\frac{\mu\Gamma(\alpha)}{\sqrt{2}\big[\Phi(1)\big]^{\alpha-1}\varphi^{\prime}(1)}&=&
6\sqrt{2\pi}\frac{\frac{3}{16}\sqrt{\sqrt{2}-1}-\frac{1}{3}\sqrt{\frac{2}{3}}\big(2\sqrt{3}-3\big)
\sqrt{2\sqrt{3}-3}}{(\sqrt{2}-1)\sqrt{\sqrt{2}-1}}\\
&\thickapprox&1.95333\\
&>&\|g\|_{\infty}.
 \end{eqnarray*}
Hence, by Theorem \ref{thmm1}, the problem \eqref{eq4.2} has a unique solution.
\end{exmp}

\end{document}